\documentclass[a4paper,12pt]{article}
\usepackage{times, url}
\usepackage{amsmath,amssymb,amsthm,amsfonts}

\newtheorem{theorem}{Theorem}
\newtheorem{remark}{Remark}
\newtheorem{conjecture}{Conjecture}

\begin{document}
\setcounter{page}{1}

\begin{center}
{\LARGE \bf Inequalities involving arithmetic functions}
\vspace{8mm}

{\Large \bf Stoyan Dimitrov}
\vspace{3mm}

Faculty of Applied Mathematics and Informatics, Technical University of Sofia\\
Blvd. St.Kliment Ohridski 8, Sofia 1756, Bulgaria \\
e-mail: \url{sdimitrov@tu-sofia.bg}

\vspace{2mm}

Department of Bioinformatics and Mathematical Modelling \\
Institute of Biophysics and Biomedical Engineering\\
Bulgarian Academy of Sciences\\
Acad. G. Bonchev Str. Bl. 105, Sofia 1113, Bulgaria \\
e-mail: \url{ xyzstoyan@gmail.com}
\vspace{2mm}
\end{center}
\vspace{10mm}

\emph{Dedicated to Professor Krassimir Atanassov on the occasion of his 70th birthday.}

\vspace{10mm}

{\bf Abstract:}
This paper presents a brief survey of the most important and the most remarkable inequalities involving the basic arithmetic functions.

\section{Introduction}
\indent

In number theory, an arithmetic function is any function $f(n)$ whose domain is the positive integers
and whose range is a subset of the complex numbers.
The arithmetic functions are very important in many  parts of theoretical and applied sciences,
and many mathematicians have devoted great interest in this field.
For example, the classical Euler totient function, and some other related functions appears almost in all
major domains of mathematics or its applications.
This review article is inspired by the book of S\'{a}ndor and  Atanassov \cite{Sandor-Atanassov}.
The number of inequalities involving arithmetic functions is too large.
That is why we do not claim to be exhaustive and comprehensive.
I apologise to anyone whose work has been overlooked.
Our main goal is to make a brief overview of the most interesting inequalities of the classical arithmetic functions.
We also propose three conjectures about the prime-counting function in subsection \ref{Conjecturesforpix}.
When presenting the mathematical results, we have tried as much as possible to observe some chronological order.
Of course we have to mention the book of S\'{a}ndor, Mitrinovi\'{c} and Crstici \cite{SandorHandbook},
providing many references on inequalities involving arithmetic functions.
We hope that the article will be useful for students, professors and researchers working in number theory.

\section{Notations, definitions and  formulas}
\indent

The letter $p$ with or without subscript will always denote prime number.
Let $n>1$ be positive integer with prime factorization
\begin{equation*}
n=p^{a_1}_1\cdots p^{a_r}_r\,.
\end{equation*}
We denote by $\omega(n)$ the number of the distinct prime factors of $n$. We have
\begin{equation*}
\omega(n)=r \quad \mbox{ and } \quad \omega(1)=0\,.
\end{equation*}
The function $\Omega(n)$ counts the total number of prime factors of $n$ honoring their multiplicity. We have
\begin{equation*}
\Omega(n)=\sum\limits_{i=1}^{r}a_i \quad \mbox{ and } \quad \Omega(1)=0\,.
\end{equation*}
We define by $d(n)$ the number of positive divisors of $n$. We have
\begin{equation*}
d(n)=\prod\limits_{i=1}^{r}(a_i+1) \quad \mbox{ and } \quad d(1)=1\,.
\end{equation*}
The function $\sigma(n)$ denotes the sum of the positive divisors of $n$.
We have
\begin{equation*}
\sigma(n)=\prod\limits_{i=1}^{r}\frac{p^{a_i+1}_i-1}{p_i-1} \quad \mbox{ and } \quad \sigma(1)=1\,.
\end{equation*}
We denote by $\varphi(n)$ the Euler totient function which is defined as
the number of positive integers not greater than $n$ that are coprime to $n$.
We have
\begin{equation*}
\varphi(n)=n \prod\limits_{i=1}^{r}\left(1-\frac{1}{p_i}\right) \quad \mbox{ and } \quad \varphi(1)=1\,.
\end{equation*}
We define the Dedekind function $\psi(n)$ by the formula
\begin{equation*}
\psi(n)=n \prod\limits_{i=1}^{r}\left(1+\frac{1}{p_i}\right) \quad \mbox{ and } \quad \psi(1)=1\,.
\end{equation*}
The prime-counting function $\pi(n)$ is the function counting the number of prime numbers less
than or equal to $n$, i.e.
\begin{equation*}
\pi(n)=\sum\limits_{p\leq n}1\,.
\end{equation*}
The function $D(n)$ denotes the arithmetic derivative. We have
\begin{align*}
&D(0)=D(1)=0\,,\\
&D(p)=1 \quad \mbox{ for any prime } p\\
&D(mn)=D(m)n+mD(n) \quad \mbox{ for any } m,n \in \mathbb{N}\,.
\end{align*}
We define by $H_n$ the $n$-th harmonic number which is the sum of the reciprocals of the first n natural numbers
\begin{equation*}
H_n=1+\frac {1}{2}+\frac {1}{3}+\cdots +\frac {1}{n}=\sum _{k=1}^{n}\frac {1}{k}\,.
\end{equation*}
The values of the sequence $H_n-\ln n$ decrease monotonically towards the limit
\begin{equation*}
\lim_{n\to \infty }\left(H_n-\ln n\right)=\gamma\,,
\end{equation*}
where $\gamma\approx 0.5772156649$ is the Euler constant. \\
We denote by $\mu(n)$ the M\"{o}bius function. We have
\begin{equation*}
\mu(n)=
\begin{cases}
(-1)^{\omega(n)}=(-1)^{\Omega(n)} & \;  \text{ if } \quad  \omega(n=\Omega(n)\,,\\
0  &  \;\text{ if }\quad    \omega(n\neq\Omega(n) \,.
\end{cases}
\end{equation*}
The function $M(n)$ denotes the Mertens function. We have
\begin{equation*}
M(n)=\sum_{k=1}^n \mu(k)\,.
\end{equation*}
As usual $\operatorname{li}(x)$ is the logarithmic integral function
\begin{equation*}
\operatorname{li}(x)=\int_{0}^{x}\frac {dt}{\ln t}\,.
\end{equation*}

\section{The divisor function {\boldmath$d(n)$}}
\indent

It is well known that for all natural numbers $n$ we have
\begin{equation*}
d(n)\leq2\sqrt{n}\,.
\end{equation*}
The proof follows from the fact that of two complementary divisors of a natural
number $n$ one is always not greater than $\sqrt{n}$.
In 1983 Nicolas and Robin \cite{Nicolas1983} established the following theorem.
\begin{theorem}
The inequality
\begin{equation*}
\frac{\log d(n)}{\log2}\leq 1.5379\frac{\log n}{\log\log n}
\end{equation*}
holds for all $n\geq3$.
\end{theorem}
\begin{proof}
See (\cite{Nicolas1983}, Theorem 1).
\end{proof}
In the same year Robin \cite{Robin1983} proved that the inequalities
\begin{align*}
&\frac{\log d(n)}{\log2}\leq \frac{\log n}{\log\log n}+1.9349\frac{\log n}{(\log\log n)^2}\,,\\
&\frac{\log d(n)}{\log2}\leq \frac{\log n}{\log\log n}+\frac{\log n}{(\log\log n)^2}+4.7624\frac{\log n}{(\log\log n)^3}\,,\\
&\frac{\log d(n)}{\log2}\leq \frac{\log n}{\log\log n-1.39177}
\end{align*}
holds for all $n\geq56$.
In 1987 S\'{a}ndor \cite{Sandor1987} deduced the results
\begin{align*}
&(d(mn))^2\geq d(m^2)\,d(n^2)  \quad \mbox{ for all } \quad  m, n = 1, 2, 3, \ldots\,,\\
&\frac{d(m^2n)\,d(k^2n) }{(d(mnk))^2}\geq \frac{d(m^2)\,d(k^2)}{(d(mk))^2}  \quad \mbox{ for all } \quad  m, n, k = 1, 2, 3, \ldots
\end{align*}
Further in 2002 S\'{a}ndor \cite{Sandor2002} got
\begin{equation*}
d(mn)\leq d(m)d(n)\quad \mbox{ for all } \quad  m, n = 1, 2, 3, \ldots\,.
\end{equation*}
In 2009 Minculete and Dicu \cite{Minculete2009} obtained the inequalities
\begin{align*}
&\frac{d(m)}{m}\leq \frac{d(n)}{n} \quad \mbox{ for all } n|m \quad  \mbox{ and } \quad  m, n = 1, 2, 3, \ldots\,.\\
&\frac{d(mn)}{mn}\leq \frac{d(m)+d(n)}{m+n}  \quad \mbox{ for all } \quad  m, n = 1, 2, 3, \ldots\,,\\
&d(mn)\leq \frac{m^2d(n)+n^2d(m)}{m+n} \quad \mbox{ for all } \quad  m, n = 1, 2, 3, \ldots\,,\\
&d(mn)\leq \frac{m^2d(n)+n^2d(m)}{m+n} \quad \mbox{ for all } \quad  m, n = 1, 2, 3, \ldots\,,\\
&d(mn)\leq m\sqrt{n}+n\sqrt{m} \quad \mbox{ for all } \quad  m, n = 1, 2, 3, \ldots\,.
\end{align*}
In the same year S\'{a}ndor and Kov\'{a}cs \cite{Sandor2009} got that for all $n\geq1$ we have
\begin{equation*}
d(n)<4\sqrt[3]{n}\,.
\end{equation*}
Other results concerning inequalities with $d(n)$ can be found in \cite{Urroz}.

\section{The sum of the positive divisors {\boldmath$\sigma(n)$}}
\indent

It is well known that
\begin{equation*}
\sigma(n)\geq n+1\,.
\end{equation*}
On the other hand for all composite numbers $n$ we have
\begin{equation*}
\sigma(n)>n+\sqrt{n}\,.
\end{equation*}
The proof follows from the fact that of two complementary divisors of a composite
number $n$ one is always not less than $\sqrt{n}$.
In 1962 Rosser and Schoenfeld \cite{Rosser1962} proved the following theorem.
\begin{theorem}\label{Rosser1962}
For every integer $n\geq3$ we have
\begin{equation*}
\sigma(n)< e^\gamma n\log\log n + \frac{2.50637n}{\log\log n}\,.
\end{equation*}
\end{theorem}
\begin{proof}
See (\cite{Rosser1962}, Theorem 15).
\end{proof}
In 1972 Annapurna \cite{Annapurna} proved that for every natural number $n\neq1, 2, 3, 4, 6, 8, 12$ we have
\begin{equation*}
\sigma(n)<\frac{6}{\pi^2}n\sqrt{n}\,.
\end{equation*}
The next year Bruckman (\cite{SandorHandbook}, p. 88) showed that if $m > 1, n > 1$ are natural numbers, then
\begin{equation*}
\sigma(mn)>\sigma(m)+\sigma(n)\,.
\end{equation*}
Further in 1984 Robin \cite{Robin1984} improved the result of Rosser and Schoenfeld in Theorem \ref{Rosser1962}.
\begin{theorem}\label{Robin1984}
The inequality
\begin{equation*}
\sigma(n) < e^\gamma n\log\log n + \frac{0.6483n}{\log\log n}
\end{equation*}
holds for all $n\geq3$.
\end{theorem}
\begin{proof}
See (\cite{Robin1984}, Theorem 2).
\end{proof}
In 2008 Sandor \cite{Sandor2008a} proved that for all composite numbers $n$ that are not the square of a prime number, we have
\begin{equation*}
\sigma(n)>n+\sqrt{n}+\sqrt[3]{n}\,.
\end{equation*}
Afterwards in 2017 Axler \cite{Axler2017} improved Robin's upper bound in Theorem \ref{Robin1984}.
\begin{theorem}\label{Axler2017} Set
\begin{equation*}
\mathcal{A} = \{1; 2; 4; 6; 8; 10; 12; 16; 18; 20; 24; 30; 36; 48; 60; 72; 120; 180; 240; 360; 2520\}
\end{equation*}
The inequality
\begin{equation*}
\sigma(n) < e^\gamma n\log\log n + \frac{0.1209n}{(\log\log n)^2}
\end{equation*}
holds for every positive integer $n\in \mathbb{N}\setminus\mathcal{A}$.
\end{theorem}
\begin{proof}
See (\cite{Axler2017}, Theorem 1).
\end{proof}
Subsequently in 2021 Aoudjit, Berkane and Dusart \cite{Aoudjit} also sharpened Robin's upper bound.
\begin{theorem}
For every integer $n\geq3$ we have
\begin{equation*}
\sigma(n)< e^\gamma n\log\log n + \frac{0.3741n}{(\log\log n)^2}\,.
\end{equation*}
\end{theorem}
\begin{proof}
See (\cite{Aoudjit}, Theorem 1.1).
\end{proof}
Recently Axler \cite{Axler2023} improved the upper bound in Theorem \ref{Axler2017}.
\begin{theorem}
For every integer $n\geq5041$ we have
\begin{equation*}
\sigma(n)< e^\gamma n\log\log n + \frac{\alpha_0n}{(\log\log n)^2}\,,
\end{equation*}
where $\alpha_0=0.0094243\,e^\gamma=0.0167853\ldots$
\end{theorem}
\begin{proof}
See (\cite{Axler2023}, Theorem 2).
\end{proof}
Also recently Axler and  Nicolas \cite{Axler-Nicolas2022} established the next theorem.
\begin{theorem}
For every integer $n\geq \exp(26 318 064 420)$ we have
\begin{equation*}
\frac{\sigma(n)}{n}< e^\gamma \log\log n+\frac{\alpha_0}{(\log\log n)^2}-\frac{2\sqrt{2}e^\gamma}{\sqrt{\log n}}
+\frac{4.143e^\gamma}{\sqrt{\log n}\log\log n}\,,
\end{equation*}
where  $\alpha_0=0.0094243\,e^\gamma=0.0167853\ldots$
\end{theorem}
\begin{proof}
See (\cite{Axler-Nicolas2022}, Theorem 1.3).
\end{proof}

\section{The Euler function {\boldmath$\varphi(n)$}}
\indent

Obviously for $n\geq2$ we have that
\begin{equation*}
\varphi(n)\leq n-1\,.
\end{equation*}
It is well known that for all composite numbers $n$ we have
\begin{equation*}
\varphi(n)\leq n-\sqrt{n}\,.
\end{equation*}
The proof follows from the formula of $\varphi(n)$ and the fact that the least prime
divisor of a composite number $n$ is not greater than $\sqrt{n}$.
As a direct consequence of the formula of $\varphi(n)$, we have
\begin{equation*}
\varphi(mn)\geq\varphi(m)\,\varphi(n)\quad \mbox{ for all } \quad  m, n = 1, 2, 3, \ldots
\end{equation*}
In 1940 Popoviciu \cite{Popoviciu} derived the result
\begin{equation*}
(\varphi(mn))^2\leq \varphi(m^2)\,\varphi(n^2)  \quad \mbox{ for all } \quad  m, n = 1, 2, 3, \ldots
\end{equation*}
In 1962 Rosser and Schoenfeld \cite{Rosser1962} proved the theorem.
\begin{theorem}
For every integer $n\geq3$ we have
\begin{equation*}
\frac{n}{\varphi(n)}< e^\gamma \log\log n + \frac{2.50637}{\log\log n}\,.
\end{equation*}
\end{theorem}
\begin{proof}
See (\cite{Rosser1962}, Theorem 15).
\end{proof}
In 1969 Hatalov\'{a} and \v{S}al\v{a}t \cite{Hatalova} showed that
\begin{equation*}
\varphi(n)>\frac{\log2}{2}\frac{n}{\log n}
\end{equation*}
for $n\geq3$.
Recently Axler and  Nicolas \cite{Axler-Nicolas2022} established the result.
\begin{theorem}\label{Axler-Nicolas2022}
For every integer $n\geq \exp(12530479255.595931)$ we have
\begin{equation*}
\frac{n}{\varphi(n)}< e^\gamma \log\log n + \frac{\alpha_0}{(\log\log n)^2}\,,
\end{equation*}
where $\alpha_0=0.0094243\,e^\gamma= 0.0167853\ldots$
\end{theorem}
\begin{proof}
See (\cite{Axler-Nicolas2022}, Theorem 1.2).
\end{proof}
\begin{remark}
Theorem \ref{Axler-Nicolas2022} is not true only for one $n$ pointed in \cite{Axler-Nicolas2022}.
\end{remark}

\section{The Dedekind function {\boldmath$\psi(n)$}}
\indent

It is well known that
\begin{equation*}
\psi(n)\geq n+1\,.
\end{equation*}
In 2011 Sol\'{e} and Planat \cite{Sole} established the following theorem.
\begin{theorem}
For  $n>30$ we have
\begin{equation*}
\psi(n)<e^\gamma n\log\log n\,.
\end{equation*}
\end{theorem}
\begin{proof}
See (\cite{Sole}, Corollary 2).
\end{proof}
Further in 2014 Mortici \cite{Mortici} showed that for $k\geq2$ we have
\begin{equation*}
(\psi(mn))^k\geq\psi(m^k)\,\psi(n^k)  \quad \mbox{ for all } \quad  m, n = 1, 2, 3, \ldots
\end{equation*}
Subsequently in 2016 S\'{a}ndor \cite{Sandor2016} among other results refined the above inequality,
proving that for any integers $m, n\geq1$ and $l\geq k\geq2$ one has
\begin{align*}
&(\psi(mn))^k\geq (mn)^{k-2}(\psi(mn))^2\geq\psi(m^k)\,\psi(n^k)\,,\\
&(\psi(mn))^k\geq (mn)^{l-k}\psi(m^l)\,\psi(n^l)\,,\\
&(mn)^{k-1}\frac{\psi(mn)}{\psi(m)\psi(n)}\leq(mn)^{k-1}\leq\frac{(\psi(mn))^k}{\psi(m)\psi(n)}\leq (\psi(m)\psi(n))^{k-1}\,.
\end{align*}
Other results concerning inequalities with $\psi(n)$ can be found in \cite{Sandor2005}.

\section{The prime-counting function {\boldmath$\pi(x)$}}

\subsection{Chebyshev-type inequalities for primes}
\indent

One of the first works on the function $\pi(x)$ is due to Chebyshev \cite{Chebyshev}.
He proved that for every sufficiently large $x$ the inequalities
\begin{equation*}
c_1\frac{x}{\log x}<\pi (x)<c_2\frac{x}{\log x}
\end{equation*}
hold. Here $c_1=\log\Big(2^\frac{1}{2}3^\frac{1}{3}5^\frac{1}{5}30^{-\frac{1}{30}}\Big)\approx 0.921292022934$
and $c_2=\frac{6}{5}c_1\approx 1.10555042752$.
The result of Chebyshev has been sharpened many times.
In 1941 Rosser \cite{Rosser1941} proved that
\begin{equation*}
\frac{x}{\log x+2}<\pi (x)<\frac{x}{\log x-4}
\end{equation*}
for $x\geq55$.
Subsequently in 1962 Rosser and Schoenfeld \cite{Rosser1962} showed that
\begin{equation*}
\frac{x}{\log x}<\pi (x)<1.25506\frac{x}{\log x}\,,
\end{equation*}
where the left inequality holds for $x\geq17$ and the right inequality holds for $x>1$, and
\begin{equation*}
\frac{x}{\log x-0.5}<\pi(x)<\frac{x}{\log x-1.5}\,,
\end{equation*}
where the left inequality holds for $x\geq67$ and the right inequality holds for $x>e^\frac{3}{2}$, and
\begin{equation*}
\frac {x}{\log x}\left(1+\frac{1}{2\log x}\right)<\pi(x)<\frac {x}{\log x}\left(1+\frac{3}{2\log x}\right)\,,
\end{equation*}
where the left inequality holds for $x\geq59$ and the right inequality holds for $x>1$.
In his 1976 book, Apostol \cite{Apostol} gives elementary proof of the inequalities
\begin{equation*}
\frac{1}{6}\frac{x}{\log x}<\pi(x)<6\frac{x}{\log x}
\end{equation*}
for $x\geq 2$.
In 1982 Nair \cite{Nair} got
\begin{equation*}
\pi (x)>(\log2)\frac{x}{\log x}
\end{equation*}
for $x\geq4$. Afterwards in 1999 Panaitopol \cite{Panaitopol1999} deduced
\begin{equation*}
\pi(x)<\frac{x}{\log x-1.12}
\end{equation*}
for $x\geq4$. In the same year Dusart \cite{Dusart1999} proved that
\begin{equation*}
\frac{x}{\log x}\left(1+\frac{1}{\log x}+\frac{1.8}{\log^2x}\right)\leq\pi(x)
\leq\frac {x}{\log x}\left(1+\frac{1}{\log x}+\frac{2.51}{\log^2x}\right)\,,
\end{equation*}
where the left inequality holds for $x \geq32299$ and the right inequality holds for $x>355991$.
Next year Panaitopol \cite{Panaitopol2000} established that
\begin{equation*}
\frac{x}{\log x-1+\frac{1}{\sqrt{\log x}}}<\pi(x)<\frac{x}{\log x-1-\frac{1}{\sqrt{\log x}}}\,,
\end{equation*}
where the left inequality holds for $x\geq59$ and the right inequality holds for $x\geq6$.
In 2006 Hassani \cite{Hassani2006} showed that
\begin{equation*}
\frac{x}{\log x-1+\frac{1}{4\log x}}<\pi(x)<\frac{x}{\log x-1+\frac{151}{100\log x}}\,,
\end{equation*}
where the left inequality holds for $x\geq3299$ and the right inequality holds for $x\geq7$.
Further in 2010 Dusart \cite{Dusart2010} obtained
\begin{equation*}
\frac{x}{\log x-1}<\pi(x)<\frac{x}{\log x-1.1}\,,
\end{equation*}
where the left inequality holds for $x\geq5393$ and the right inequality holds for $x\geq601841$.
In 2016 Berkane and Dusart \cite{Berkane} proved that
\begin{equation*}
\frac{x}{\log x}\left(1+\frac{1}{\log x}+\frac{2}{\log^2x}+\frac {5.2}{\log^3x}\right)\leq\pi(x)
\leq\frac{x}{\log x}\left(1+\frac{1}{\log x}+\frac{2}{\log^2x}+\frac {7.57}{\log^3x}\right)\,,
\end{equation*}
where the left inequality holds for $x>3596143$ and the right inequality holds for $x>110118914$.
Subsequently Dusart \cite{Dusart2018} established that
\begin{equation*}
\frac {x}{\log x}\left(1+\frac{1}{\log x}\right)\leq\pi(x)\leq\frac {x}{\log x}\left(1+\frac{1.2762}{\log x}\right)\,,
\end{equation*}
where the left inequality holds for $x \geq599$ and the right inequality holds for $x>1$, and
\begin{equation*}
\frac{x}{\log x}\left(1+\frac{1}{\log x}+\frac{2}{\log^2x}\right)\leq\pi(x)
\leq\frac {x}{\log x}\left(1+\frac{1}{\log x}+\frac{2.53816}{\log^2x}\right)\,,
\end{equation*}
where the left inequality holds for $x \geq88789$ and the right inequality holds for $x>1$, and
\begin{equation*}
\pi(x)\leq\frac{x}{\log x}\left(1+\frac{1}{\log x}+\frac{2}{\log^2x}+\frac {7.59}{\log^3x}\right)
\end{equation*}
for $x>1$.
Other results concerning inequalities with $\pi(x)$ can be found in the papers of
Axler \cite{Axler2016}, \cite{Axler2018}, \cite{Axler2018a}, \cite{Axler2022}.
It was shown in 2016 by Trudgian \cite{Trudgian2016} that
\begin{equation*}
\big |\pi (x)-\operatorname {li}(x)\big |\leq 0.2795\frac {x}{(\log x)^{3/4}}e^{-\sqrt {\frac {\log x}{6.455}}}
\end{equation*}
for $x\geq 229$.
The result of Trudgian was sharpened by Mossinghoff and Trudgian \cite{Mossinghoff2015} with the explicit upper bound
\begin{equation*}
\big |\pi (x)-\operatorname {li}(x)\big |\leq 0.2593\frac {x}{(\log x)^{3/4}}e^{-\sqrt {\frac {\log x}{6.315}}}
\end{equation*}
for $x\geq 229$.
Ramanujan \cite{Berndt} is credited with the inequality
\begin{equation}\label{Ramanujan-inequality}
\pi^2(x)<\frac{ex}{\log x}\pi\left(\frac{x}{e}\right)\,,
\end{equation}
which is valid for all sufficiently large values of $x$.
The inequality \eqref{Ramanujan-inequality} is called Ramanujan's prime counting inequality.
A further contribution to this problem was proposed by Dudek and Platt \cite{Dudek2015} in 2015
when they proved that inequality \eqref{Ramanujan-inequality} holds unconditionally for all
$x\geq \textmd{exp(9658)}$.
The result of Dudek and Platt was then successively improved by Dudek \cite{Dudek2016} to $x\geq \textmd{exp(9394)}$,
by Axler \cite{Axler2018a}  to $x\geq \textmd{exp(9032)}$,
by Platt and Trudgian  \cite{Platt2021} to $x\geq \textmd{exp(3915)}$,
by Cully-Hugill and Johnston \cite{Cully-Hugill} to $x\geq \textmd{exp(3604)}$,
by Johnston and Yang  \cite{Johnston2022} to $x\geq \textmd{exp(3361)}$
and the best result up to now belongs to Axler \cite{Axler2022a} with $x\geq \textmd{exp(3158.442)}$.
In his recent paper Hassani \cite{Hassani2021} studies inequalities analogous to Ramanujan's inequality \eqref{Ramanujan-inequality}.

\subsection{The second Hardy-Littlewood conjecture}
\indent

One of the most famous hypotheses in number theory is the second Hardy-Littlewood conjecture,
which was published in 1923 in their paper \cite{Hardy-Littlewood}.
\begin{conjecture}\label{HLconjecture} (Second Hardy-Littlewood conjecture) Let $x, y \geq2$. Then
\begin{equation}\label{H-Lconjecture}
\pi(x+y)\leq\pi(x)+\pi(y)\,.
\end{equation}
\end{conjecture}
This is still an open problem.
Inequality \eqref{H-Lconjecture} has attracted the attention of numerous mathematicians, who presented various related
results. There exist many inequalities related to \eqref{H-Lconjecture}. We mention a few of them.
In 1934 Ishikawa \cite{Ishikawa} deduced
\begin{equation*}
\pi(x)+\pi(y)\leq\pi(xy)\,,
\end{equation*}
where $x, y\geq 5$.
Schinzel \cite{Schinzel} in 1961 proved that for all positive integers $x, y$ with $\min(x, y) \leq146$
the inequality \eqref{H-Lconjecture} holds.
Further in 1962 Rosser and Schoenfeld \cite{Rosser1962} established that
\begin{equation*}
\frac{3x}{5\log x}<\pi(2x)-\pi(x)<\frac{7x}{5\log x}\,,
\end{equation*}
where the left inequality holds for $x \geq\frac{41}{2}$ and the right inequality holds for $x>1$.
In the same year Segal \cite{Segal} proves the inequality \eqref{H-Lconjecture} for $x+y\leq101081$.
In 1966 Rosser and Schoenfeld \cite{Rosser1966} proved that Conjecture \ref{HLconjecture} is true when $x=y$.
This is the famous inequality
\begin{equation}\label{Landau-inequality}
\pi(2x)\leq2\pi(x)
\end{equation}
which is suggested by Landau \cite{Landau}.
Thanks to Karanikolov \cite{Karanikolov} we have that
\begin{equation*}
\pi(ax)\leq a\pi(x)
\end{equation*}
for $a\geq e^{1/4}$ and $x\geq364$.
In 1973 Montgomery and Vaughan \cite{Montgomery} using the large sieve obtained
\begin{equation*}
\pi(x+y)\leq\pi(x)+2\pi(y)
\end{equation*}
for $x\geq1$, $y\geq2$.
Udrescu \cite{Udrescu} in 1975 derived that if $0<a\leq1$ and $ax\leq y\leq x$ then
the inequality \eqref{H-Lconjecture} holds for $x$ and $y$ sufficiently large.
Afterwards in 1979 Panaitopol \cite{Panaitopol1979} proved that if $a<b<c<d$ are positive numbers
with $a+d=b+c$ then, if $x$ exceeds an (unspeciﬁed) margin, we have
\begin{equation}\label{Panaitopol1979}
\pi(ax)+\pi(dx)<\pi(bx)+\pi(cx)\,.
\end{equation}
In 1998 Gordon and Rodemich \cite{Gordon} showed that \eqref{H-Lconjecture} is true
when $2 \leq \min(x, y) \leq1731$.
In the same year Dusart \cite{Dusart1998} obtained  that the inequality \eqref{H-Lconjecture} holds
for $x\leq y\leq\frac{7}{5}x(\log x)\log\log x$, and
\begin{equation*}
\frac{1}{2}[\pi(x)+\pi(2x)]\leq\pi(3x)\leq\pi(x)+\pi(2x)
\end{equation*}
for all integers $x\geq2$, and
\begin{equation*}
\pi(kx)\leq k\pi(x)
\end{equation*}
for all integers $x\geq3$, $k\geq1$, and
\begin{equation*}
\pi(2x)-\pi(x)\leq\frac{x}{\log x}
\end{equation*}
for all integers $x\geq2$. Again in the same year Panaitopol \cite{Panaitopol1998} proved that
\begin{equation*}
\pi(xy)\geq\pi(x)\pi(y)
\end{equation*}
for $x, y\geq2$ with the exceptions of $(x, y)=(7, 5); (7, 7)$.
Further in 2000 Panaitopol \cite{Panaitopol2000} established that the inequality \eqref{H-Lconjecture}
holds for $ax\leq y\leq x$,\, $x>e^{9a^{-2}}$,\, $0<a\leq1$, and
\begin{equation*}
\pi(ax)<a\pi(x)
\end{equation*}
for $a>1$, $x>e^{4(\log a)^{-2}}$, and
\begin{equation*}
\pi(x+y)\leq2\pi(x/2)+\pi(y)
\end{equation*}
for $2\leq y\leq x$,\, $x\geq6$, and
\begin{equation*}
\pi(x+y)\leq2[\pi(x/2)+\pi(y/2)]
\end{equation*}
for all integers $x, y\geq4$.
In the same year Panaitopol \cite{Panaitopol2000a} proved that for $\alpha<1$ there exists
$n=n(\alpha)$ (explicitly computable) such that for all $x, y\geq n$ one has
\begin{equation*}
\pi^\alpha(x+y)\leq\pi^\alpha(x)+\pi^\alpha(y)\,.
\end{equation*}
Again Panaitopol \cite{Panaitopol2001} in 2001 got
\begin{equation*}
\pi(x+y)\leq\pi(x)+\pi(y)+\pi(x-y)\,,
\end{equation*}
where $2\leq y\leq x$ and
\begin{equation}\label{Panaitopol2001}
\pi^2(x+y)\leq2[\pi^2(x)+\pi^2(y)]
\end{equation}
for all integers $x, y \geq 2$, and
\begin{equation*}
2\frac{\pi(x+y)}{x+y}\leq\frac{\pi(x)}{x}+\frac{\pi(y)}{y}
\end{equation*}
except the cases $(x, y)=(3, 2)$, $(5, 2)$.
In the same year Panaitopol \cite{Panaitopol2001a} showed that the inequality \eqref{H-Lconjecture}
holds for all integers $x, y \geq 2$ that fulfill the conditions $\frac{y}{29}\leq x\leq y$.
Thanks to Garunk\v{s}tis \cite{Garun} from 2002 we have that
\begin{equation}\label{Garun}
\pi(x+y)\leq1.11\pi\left(\frac{x}{1.11}\right)+\pi(y)\,,
\end{equation}
where $7\leq y\leq x$.
Subsequently in 2003 Mincu \cite{Mincu} improved the inequality \eqref{Garun} to
\begin{equation*}
\pi(x+y)\leq1.04\pi\left(\frac{x}{1.04}\right)+\pi(y)\,,
\end{equation*}
where $17\leq y\leq x$. Moreover Mincu \cite{Mincu} found an explicit margin for $x$ from which
on \eqref{Panaitopol1979} checks up and he sharpened \eqref{Landau-inequality} to
\begin{equation*}
\pi\big(2x+\pi(x)\big)\leq2\pi(x)
\end{equation*}
for all integers $x\geq37$, and also established that
\begin{equation*}
\pi(x)+\pi(4x)<\pi(2x)+\pi(3x)
\end{equation*}
for all real $x\geq284.5$.
Further in 2012 Minculete \cite{Minculete2012} obtained
\begin{equation*}
b\pi(ax)<a\pi(bx)
\end{equation*}
for every positive integers $a, b, x$ with $bx \geq5393$,\, $ax\geq4$,\, $a\geq be^{0.12}$, and
\begin{equation*}
(x+y)\pi(x+y)<2[x\pi(x)+y\pi(y)]
\end{equation*}
for all integers $x, y \geq67$, and
\begin{equation*}
\pi^2(xy)<\frac{1}{4}xy\pi(x)\pi(y)
\end{equation*}
for all integers $x, y \geq1525$, and
\begin{equation*}
(1-a)\pi(x+y)<(1+a)\pi(x)+\pi(y)
\end{equation*}
for all integers $10544111 \leq y\leq x$ and $a\geq0.00042$.
In addition, Minculete \cite{Minculete2012} formulated the following hypothesis.
\begin{conjecture}
There exists sufficiently large natural number $n_0$, such that for all $x, y \geq n_0$, we have the inequality
\begin{equation*}
\pi^2(xy)\leq\pi(x^2)\pi(y^2)\,.
\end{equation*}
\end{conjecture}
Inspired by Panaitopol's result \eqref{Panaitopol2001} in 2013 Alzer \cite{Alzer2013} proved that
\begin{equation*}
\pi^2(x)+\pi^2(y)\leq\frac{5}{4}\pi^2(x+y)
\end{equation*}
holds for all integers $x, y \geq 2$.
In 2022 Alzer, Kwong and S\'{a}ndor \cite{Alzer2022} showed that
\begin{equation*}
\pi(x)\pi(y)\leq\frac{9}{16}\pi^2(x+y)
\end{equation*}
is valid for all integers $x, y \geq2$ with equality only for $x=5$.
The constant factor $\frac{9}{16}$ is the best possible. The special case $x=y$ leads to
\begin{equation*}
\pi(2x)\geq\frac{4}{3}\pi(x)\,,
\end{equation*}
which is a converse of the Landau inequality \eqref{Landau-inequality}.
Moreover in \cite{Alzer2022} we find the results
\begin{equation*}
\pi(x+y)\leq2\Big[\pi(x)^x\pi(y)^y\Big]^\frac{1}{x+y}
\end{equation*}
and
\begin{equation*}
\left(2\frac{\pi(x+y)}{x+y}\right)^s\leq \left(\frac{\pi(x)}{x}\right)^s+\left(\frac{\pi(y)}{y}\right)^s\,.
\end{equation*}
The latter inequality holds for all integers $x, y \geq2$ if and only if $0<s\leq s_0=0.94745\ldots$.
Here, $s_0$ is the only positive solution of the equation
\begin{equation*}
\left(\frac{16}{7}\right)^t-\left(\frac{6}{5}\right)^t=1\,.
\end{equation*}
Subsequently S\'{a}ndor \cite{Sandor2021} deduced that for all $x, y\geq 2$ we have
\begin{equation*}
\pi(x+y)\geq \frac{2}{3}[\pi(x)+\pi(y)]
\end{equation*}
and
\begin{equation*}
\pi^2(x+y)\geq \frac{8}{7}[x\pi(x)+y\pi(y)]\,.
\end{equation*}
Further S\'{a}ndor \cite{Sandor2022a} got
\begin{equation*}
\sqrt{3\pi(x+y)}\geq\sqrt{\pi(x})+\sqrt{\pi(y)}\,,
\end{equation*}
\begin{equation*}
\sqrt{x\pi(x)}+\sqrt{y\pi(y)}\leq\frac{\sqrt{6}}{2}\sqrt{(x+y)\pi(x+y)}
\end{equation*}
for all $x, y\geq 1$, and
\begin{equation*}
\sqrt{\pi(x+y)}\leq\sqrt{\pi(x})+\sqrt{\pi(y)}
\end{equation*}
for all $x, y\geq 2$, and
\begin{equation*}
\sqrt{2\pi(x+y)}\geq\sqrt{\pi(x})+\sqrt{\pi(y)}\,,
\end{equation*}
\begin{equation*}
\sqrt{2\pi(x+y)}\leq\sqrt{\pi(x})+\sqrt{\pi(y)}
\end{equation*}
for infinitely many $(x, y)$, and
\begin{equation*}
(x+y)\sqrt{\pi(x+y)}\leq x\sqrt{2\pi(x})+y\sqrt{2\pi(y)}
\end{equation*}
for all $x, y\geq 2$, and
\begin{equation*}
\pi(x+y)\sqrt{x+y}\leq \pi(x)\sqrt{2x}+\pi(y)\sqrt{2y}
\end{equation*}
for any $2\leq y\leq x$, except for the values $(x, y)=(4, 3), (10, 9)$, and
\begin{equation*}
\pi(x^2+y^2)\geq \frac{5}{6}[\pi(x^2)+\pi(y^2)]\,,
\end{equation*}
\begin{equation*}
\sqrt{\frac{12}{5}\pi(x^2+y^2)}\geq\sqrt{\pi(x^2)}+\sqrt{\pi(y^2)}
\end{equation*}
for all $x, y\geq 2$.
Very recently S\'{a}ndor \cite{Sandor2023} has improved his results in \cite{Alzer2022},
\cite{Sandor2021} and \cite{Sandor2022a} showing that for all $x, y\geq 2$ we have
\begin{equation*}
\pi(x+y)\geq \frac{3}{4}[\pi(x)+\pi(y)]
\end{equation*}
with the exceptions of $(x, y)=(7, 3); (5, 5); (7, 5); (23, 13); (19, 17)$, and
\begin{equation*}
\pi(2x)\geq\frac{3}{2}\pi(x)\,,
\end{equation*}
for any $x\neq5$, $x\geq2$, and
\begin{equation*}
\pi(x)+\pi(y)\leq\pi(x)\pi(y)\,,
\end{equation*}
for any $x,y\geq3$, and
\begin{equation*}
\pi(x)+\pi(y)\leq\frac{2}{3}\pi(x)\pi(y)\,,
\end{equation*}
with the exceptions of $(x, y)=(3, 3); (4, 3); (4, 4)$, and
\begin{equation*}
\pi(x+y)\leq\frac{x}{y}\pi(x)+\pi(y)\,,
\end{equation*}
\begin{equation*}
\pi(x+y)\leq 2\sqrt{\pi(x)^{x/y}\pi(y)}\leq\pi(x)^{x/y}+\pi(y)
\end{equation*}
for $2\leq y\leq x$, and
\begin{equation*}
\frac{5}{4}\leq\frac{\pi(\pi(2x))}{\pi(\pi(x))}\leq2
\end{equation*}
for $x\geq3$ except for $x=5$ on the left side, and
\begin{equation*}
\frac{3}{2}\leq\frac{\pi(\pi(3x))}{\pi(\pi(x))}\leq3
\end{equation*}
for $x\geq3$.

\subsection{Conjectures about {\boldmath$\pi(x)$}}\label{Conjecturesforpix}
\indent

Using the analogy of the Young, H\"{o}lder and Minkowski inequalities the author formulates the following three hypotheses.
\begin{conjecture}
Let $\frac{1}{p}+\frac{1}{q}=1$ with $p, q>1$.
There exists sufficiently large natural number $n_0$, such that for all $x, y \geq n_0$ the inequality
\begin{equation*}
\pi(xy)\leq\frac{\pi(x^p)}{p}+\frac{\pi(x^q)}{q}
\end{equation*}
holds.
\end{conjecture}

\begin{conjecture}
Let $\frac{1}{p}+\frac{1}{q}=1$ with $p, q>1$.
There exists sufficiently large natural number $n_0$, such that for all $x, y \geq n_0$ the inequality
\begin{equation*}
\pi(xy)\leq[\pi(x^p)]^\frac{1}{p}[\pi(y^q)]^\frac{1}{q}
\end{equation*}
holds.
\end{conjecture}

\begin{conjecture}
Let $p>1$.
There exists sufficiently large natural number $n_0$, such that for all $x, y \geq n_0$ the inequality
\begin{equation*}
[\pi\big((x+y)^p\big)]^\frac{1}{p}\leq[\pi(x^p)]^\frac{1}{p}+[\pi(y^p)]^\frac{1}{p}
\end{equation*}
holds.
\end{conjecture}

\section{The arithmetic derivative {\boldmath$D(n)$}}
\indent

In 1961 Barbeau \cite{Barbeau} obtained upper and lower bounds for $D(n)$ and found that
\begin{equation*}
D(n)\leq\frac {n\log_2n}{2}
\end{equation*}
and
\begin{equation*}
D(n)\geq\Omega(n)n^\frac {\Omega(n)-1}{\Omega (n)}
\end{equation*}
for $n\geq1$.

Afterwards in 2011 Dahl, Olsson and Loiko \cite{Dahl} showed that
\begin{equation*}
D(n)\leq \frac{n\log_pn}{p}\,,
\end{equation*}
where $p$ is the smallest prime factor of $n$ and equality holds when $n$ is a power of $p$.

\section{The {\boldmath$n$}-th harmonic number {\boldmath$H_n$}}
\indent

It is easy to see that
\begin{equation*}
\gamma+\log n<H_n<1+\log n\,,
\end{equation*}
where $n>1$ and $\gamma$ is Euler's constant.
Over the years, the upper and lower bounds of $H_n$ have been improved many times.
In 1971 Tims and Tyrrell \cite{Tims} proved the following theorem.
\begin{theorem}
For all integers $n\geq2$ we have
\begin{equation*}
\gamma+\log n+\frac{1}{2(n+1)}<H_n<\gamma+\log n+\frac{1}{2(n-1)}\,.
\end{equation*}
\end{theorem}
\begin{proof}
See (\cite{Tims}, , Theorem 1, Corollary).
\end{proof}
Next step belongs to Young \cite{Young}. In 1991 he proved the next result.
\begin{theorem}\label{Young}
For all natural numbers $n$ we have
\begin{equation*}
\gamma+\log n+\frac{1}{2(n+1)}<H_n<\gamma+\log n+\frac{1}{2n}\,.
\end{equation*}
\end{theorem}
\begin{proof}
See (\cite{Young}, Theorem 1).
\end{proof}
Subsequently T\'{o}th \cite{Toth1}, \cite{Toth2} improved the result of Young.
\begin{theorem}
For all natural numbers $n$ we have
\begin{equation*}
\gamma+\log n+\frac{1}{2n+\frac{2}{5}}<H_n<\gamma+\log n+\frac{1}{2n+\frac{1}{3}}\,.
\end{equation*}
\end{theorem}
\begin{proof}
See (\cite{Toth1}, \cite{Toth2}).
\end{proof}
Further in 1993 DeTemple \cite{DeTemple} established the following theorem.
\begin{theorem}\label{DeTemple}
For all natural numbers $n$ we have
\begin{equation*}
\gamma+\log\left(n+\frac{1}{2}\right)+\frac{1}{24(n+1)^2}<H_n<\gamma+\log\left(n+\frac{1}{2}\right)+\frac{1}{24n^2}\,.
\end{equation*}
\end{theorem}
\begin{proof}
See (\cite{DeTemple}, Theorem 1).
\end{proof}
In 1995 Anderson et al. \cite{Anderson} derived result which sharpens the lower bound given in Theorem \ref{Young} if $n\leq5$.
\begin{theorem}
For all natural numbers $n$ we have
\begin{equation*}
\gamma+\log n+\frac{1-\gamma}{n}<H_n<\gamma+\log n+\frac{1}{2n}\,.
\end{equation*}
\end{theorem}
\begin{proof}
See (\cite{Anderson}, Theorem 3.1).
\end{proof}
Next step belongs to Alzer \cite{Alzer1998}. In 1998 he established the following result.
\begin{theorem}
For all natural numbers $n$ we have
\begin{equation*}
\gamma+\log n+\frac{1}{2n+\frac{2\gamma-1}{1-\gamma}}<H_n<\gamma+\log n+\frac{1}{2n+\frac{1}{3}}\,.
\end{equation*}
\end{theorem}
\begin{proof}
See (\cite{Alzer1998}, Theorem 3).
\end{proof}
In 1999 Negoi \cite{Negoi} formulated and proved the following theorem.
\begin{theorem}
For all natural numbers $n$ we have
\begin{equation*}
\gamma+\log\left(n+\frac{1}{2}+\frac{1}{24n}\right)-\frac{1}{48n^3}<H_n
<\gamma+\log\left(n+\frac{1}{2}+\frac{1}{24n}\right)-\frac{1}{48(n+1)^3}\,.
\end{equation*}
\end{theorem}
\begin{proof}
See (\cite{Negoi}, Theorem 1).
\end{proof}
In 2005 Batir \cite{Batir2005} provided remarkable upper and lower bounds for $H_n$.
\begin{theorem}
For any integer $n\geq1$ the following inequalities involving harmonic numbers hold.
\begin{equation*}
\gamma+\log\left(n+\frac{1}{2}\right)<H_n\leq \gamma+\log(n-1+e^{1-\gamma})\,,
\end{equation*}
\begin{equation*}
\log\left(\frac{\pi^2}{6}\right)-\log\big(e^{1/(n+1)}-1\big)<H_n<\gamma-\log\big(e^{1/(n+1)}-1\big)\,.
\end{equation*}
\end{theorem}
\begin{proof}
See (\cite{Batir2005}, Corollary 2.2).
\end{proof}
Next year  Alzer \cite{Alzer2006} refines the result of Batir.
\begin{theorem}
For all natural numbers $n$ we have
\begin{equation*}
1+\log\big(\sqrt{e}-1\big)-\log\big(e^{1/(n+1)}-1\big)\leq H_n<\gamma-\log\big(e^{1/(n+1)}-1\big)\,.
\end{equation*}
\end{theorem}
\begin{proof}
See (\cite{Alzer2006}, Theorem 1).
\end{proof}
In 2009 Chen \cite{Chen2009} obtained result which sharpens the upper bound given in Theorem \ref{DeTemple}.
\begin{theorem}
For all natural numbers $n$ we have
\begin{equation*}
H_n<\gamma+\log\left(n+\frac{1}{2}\right)+\frac{1}{24\left(n+\frac{1}{2}\right)^2}\,.
\end{equation*}
\end{theorem}
\begin{proof}
See (\cite{Chen2009}, Remark 4).
\end{proof}
Further in 2010 Chen \cite{Chen2010} established  the following result.
\begin{theorem}
For all natural numbers $n$ we have
\begin{equation*}
\gamma+\log\left(n+\frac{1}{2}\right)+\frac{1}{24(n+a)^2}
\leq H_n<\gamma+\log\left(n+\frac{1}{2}\right)+\frac{1}{24(n+b)^2}
\end{equation*}
with the best possible constants
\begin{equation*}
a=\frac{1}{\sqrt{24[1-\gamma-\log(3/2)]}} \quad \mbox{ and } \quad   b=\frac{1}{2} \,.
\end{equation*}
\end{theorem}
\begin{proof}
See (\cite{Chen2010}, Theorem 1).
\end{proof}

In 2011 Alzer \cite{Alzer2011} proved several interesting results.
\begin{theorem}
For all integers $n\geq2$ we have
\begin{equation*}
\alpha\frac{\log(\log n+\gamma)}{n^2}<H^{1/n}_n-H^{1/(n+1)}_{n+1}<\frac{\log(\log n+\gamma)}{n^2}
\end{equation*}
with the best possible constant factors
\begin{equation*}
\alpha=\frac{6\sqrt{6}-2\sqrt[3]{396}}{3\log(\log 2+\gamma)}
\end{equation*}
\end{theorem}
\begin{proof}
See (\cite{Alzer2011}, Theorem 3.1).
\end{proof}

\begin{theorem}
For all integers $n\geq4$ we have
\begin{equation*}
H^{1/n}_n<\sqrt{H^{1/(n-1)}_{n-1}H^{1/(n+1)}_{n+1}}\,.
\end{equation*}
\end{theorem}
\begin{proof}
See (\cite{Alzer2011}, Theorem 3.2).
\end{proof}

\begin{theorem}
For all integers $m\geq2$ and $n\geq2$ we have
\begin{equation*}
2<\frac{H^{1/m}_m+H^{1/n}_n}{H^{1/(m+n)}_{m+n}}<1+\frac{\sqrt{6}}{2}\,.
\end{equation*}
\end{theorem}
\begin{proof}
See (\cite{Alzer2011}, Theorem 3.4).
\end{proof}

\begin{theorem}
For all integers $m\geq2$ and $n\geq2$ we have
\begin{equation*}
\frac{1}{H^{1/(m+n)}_{m+n}}<\frac{4}{H^{1/m}_m+H^{1/n}_n}<\frac{1}{H^{1/m}_m}+\frac{1}{H^{1/n}_n}\,.
\end{equation*}
\end{theorem}
\begin{proof}
See (\cite{Alzer2011}, Remark 3.5).
\end{proof}

\begin{theorem}
Let $k, m, n \in \mathbb{N}\setminus\{1, 2\}$. If $m<n$, then
\begin{equation*}
H^{1/(m+k)}_{m+k}-H^{1/m}_m<H^{1/(n+k)}_{n+k}-H^{1/n}_n\,.
\end{equation*}
\end{theorem}
\begin{proof}
See (\cite{Alzer2011}, Theorem 3.6).
\end{proof}

\begin{theorem}
For all integers $n\geq1$ we have
\begin{equation*}
\alpha\leq \exp(H_{n+1})-\exp(H_n)<\beta
\end{equation*}
with the best possible bounds
\begin{equation*}
\alpha=e(\sqrt{e}-1) \quad \mbox{ and } \quad \beta=e^\gamma\,.
\end{equation*}
\end{theorem}
\begin{proof}
See (\cite{Alzer2011}, Theorem 3.7).
\end{proof}
In the same year Batir \cite{Batir2011} proved the following theorem.
\begin{theorem}
For any integer $n\geq1$ we have
\begin{equation*}
\gamma+\frac{1}{2}\log\Big(n^2+n+e^{2(1-\gamma)}-2\Big)<H_n\leq \gamma+\frac{1}{2}\log\left(n^2+n+\frac{1}{3}\right)\,.
\end{equation*}
\end{theorem}
\begin{proof}
See (\cite{Batir2011}, Corollary 2.2).
\end{proof}
Again in 2011 Guo and Qi \cite{Guo} obtained the following result.
\begin{theorem}
For all natural numbers $n$ we have
\begin{equation*}
\gamma+\log n+\frac{1}{2n}-\frac{1}{12n^2+\frac{2(7-12\gamma)}{2\gamma-1}}\leq H_n
<\gamma+\log n+\frac{1}{2n}-\frac{1}{12n^2+\frac{6}{5}}\,.
\end{equation*}
\end{theorem}
\begin{proof}
See (\cite{Guo}, Theorem 1).
\end{proof}
Other results concerning inequalities with $H_n$ can be found in
\cite{Chen2012},  \cite{Elezovic}, \cite{Niu}, \cite{Sintamarian},

\section{The Mertens function {\boldmath$M(n)$}}
\indent

In 1995 Marraki \cite{Marraki} obtained upper bounds for $M(n)$ and found that
\begin{equation*}
{\begin{aligned}|M(x)|&\leq{\frac {12590292\cdot x}{(\log x)^{236/75}}}\quad {\text{ for }} \quad x>\exp(12282.3)
\\|M(x)|&\leq{\frac {0.6437752\cdot x}{\log x}}\quad {\text{ for }} \quad x>1\,.\end{aligned}}
\end{equation*}
Other results concerning inequalities with $M(n)$ can be found in \cite{Cohen} and \cite{Ramare}.

\section{Prime omega functions {\boldmath$\omega(n)$} and {\boldmath$\Omega(n)$}}
\indent

Obviously
\begin{equation*}
\omega(n)\leq\Omega(n)\,.
\end{equation*}
In 1983 Robin \cite{Robin1983a} obtained the following results.
\begin{theorem}
For $n\geq3$ we have
\begin{equation*}
\omega(n)\leq 1.38402\frac{\log n}{\log\log n}\,.
\end{equation*}
\end{theorem}
\begin{proof}
See (\cite{Robin1983a}, Theorem 11).
\end{proof}

\begin{theorem}
For $n\geq3$ we have
\begin{equation*}
\omega(n)\leq\frac{\log n}{\log\log n}+1.45743\frac{\log n}{(\log\log n)^2}\,.
\end{equation*}
\end{theorem}
\begin{proof}
See (\cite{Robin1983a}, Theorem 12).
\end{proof}

\begin{theorem}
For $n\geq26$ we have
\begin{equation*}
\omega(n)\leq \frac{\log n}{\log\log n-1.1714}\,.
\end{equation*}
\end{theorem}
\begin{proof}
See (\cite{Robin1983a}, Theorem 13).
\end{proof}

\begin{theorem}
For $n\geq3$ we have
\begin{equation*}
\omega(n)\leq\frac{\log n}{\log\log n}+\frac{\log n}{(\log\log n)^2}+2.89726\frac{\log n}{(\log\log n)^3}\,.
\end{equation*}
\end{theorem}
\begin{proof}
See (\cite{Robin1983a}, Theorem 16).
\end{proof}
In 2022 Alzer and Kwong \cite{Alzer2022a} proved the following theorem.
\begin{theorem}
For $n\geq2$ we have
\begin{equation*}
\omega(2n)\leq2\omega(n)\,.
\end{equation*}
Equality holds if and only if $n=p^k$, where $p$ is an odd prime number and $k$ is a positive integer.
\end{theorem}
\begin{proof}
See (\cite{Alzer2022a}, Lemma 2).
\end{proof}

\section{The functions {\boldmath$\varphi(n)$} and {\boldmath$\sigma(n)$}}
\indent

The most famous inequality connecting $\varphi(n)$ and $\sigma(n)$ is contained in the following theorem.
\begin{theorem}
We have
\begin{equation*}
\frac{6}{\pi^2}<\frac{\varphi(n)\,\sigma(n)}{n^2}<1 \,.
\end{equation*}
\end{theorem}

\begin{proof}
See (\cite{Hardy-Wright}, Theorem 329).
\end{proof}
In 1968 Oppenheim \cite{Oppenheim} showed that
\begin{equation*}
\varphi\left(n\left[\frac{\sigma(n)}{n}\right]\right)\leq n\,.
\end{equation*}
Further in 1987 Atanassov \cite{Atanassov1987} and S\'{a}ndor \cite{Sandor1987} established that if $n\geq2$, then
\begin{equation*}
\sigma(n)^{\varphi(n)}<n^n<\varphi(n)^{\sigma(n)}\,.
\end{equation*}
In 1988 Atanassov \cite{Atanassov1988} proved that if $n$ is odd, then
\begin{equation*}
\sigma(n)\leq \varphi(n)P(n)
\end{equation*}
and if $n$ is even, then
\begin{equation*}
\sigma(n)\leq 4\varphi(n)P(n)\,,
\end{equation*}
where $P(n)$ denotes the greatest prime factor of $n$.
Further in 2008  Atanassov \cite{Atanassov2008} showed that
\begin{equation*}
\frac{6}{\pi^2}\varphi^2(n)>\sigma(n)\sqrt{n}
\end{equation*}
for all odd numbers except for the values $n=3, 5, 9, 15, 27, 45, 75$, and
\begin{equation*}
\frac{18\sqrt{2}}{\pi^2}\varphi^2(n)>\sigma(n)\sqrt{n}
\end{equation*}
for all even numbers $n>6$, and
\begin{equation*}
\varphi^3(n)>\sigma^2(n)
\end{equation*}
for all $n\geq5$ except for the values $n=6, 8, 12, 16, 18, 24$.
In 2019 S\'{a}ndor and Atanassov \cite{Sandor2019} deduced that if $n\geq2$, then
\begin{align*}
&\varphi(n)+\sigma(n)\geq2n\,,\\
&\varphi(n)^{\varphi(n)}\sigma(n)^{\sigma(n)}>n^{\varphi(n)+\sigma(n)}\,,\\
&\varphi(n)^{\sigma(n)}\sigma(n)^{\varphi(n)}<n^{\varphi(n)+\sigma(n)}\,.
\end{align*}
Other results concerning inequalities with $\varphi(n)$ and $\sigma(n)$ can be found in
\cite{Atanassov1991a}, \cite{Atanassov1991b}, \cite{Atanassov1991c}, \cite{Atanassov2006}.

\section{The functions {\boldmath$\varphi(n)$} and {\boldmath$d(n)$}}
\indent

In 1967 Sivaramakrishnan \cite{Sivaramakrishnan1967} proved that for all natural numbers
\begin{equation*}
\varphi(n)\,d(n)\geq n\,.
\end{equation*}
Further in 1972 Porubsky \cite{Porubsky} showed that if $n\neq4$, then
\begin{equation*}
\varphi(n)\,d^2(n)\leq n^2\,.
\end{equation*}
Thanks to P\'{o}lya and Szeg\"{o} \cite{Polya} from 1976 we have that if $n>30$, then
\begin{equation*}
\varphi(n)>d(n)\,.
\end{equation*}
In 1987 S\'{a}ndor \cite{Sandor1987} proved the inequality
\begin{equation*}
\varphi\left(n\left[\frac{n}{d(n)}\right]\right)\leq \varphi^2(n)\,.
\end{equation*}
Further in 1989 S\'{a}ndor \cite{Sandor1989} proved that for all natural numbers
\begin{equation*}
\varphi(n)\,d(n)\geq \varphi(n)+n-1\,.
\end{equation*}
In 2010 S\'{a}ndor \cite{Sandor2010} showed that if $n\neq4$ is composite, then
\begin{equation*}
\varphi(n)+d(n)<n
\end{equation*}
and that for $n\geq2$ we have
\begin{equation*}
\varphi(n)+d(n)<n+1\,.
\end{equation*}
Other results concerning inequalities with $\varphi(n)$ and $d(n)$ can be found in
\cite{Sandor2020}, \cite{Sandor2022}.

\section{The functions {\boldmath$d(n)$} and {\boldmath$\sigma(n)$}}
\indent

In 1965 Sivaramakrishnan and Venkataraman \cite{Sivaramakrishnan1965} established that for all natural numbers
\begin{equation*}
\frac{\sigma(n)}{d(n)}\geq\sqrt{n}\,.
\end{equation*}
On the other hand it is known that for all natural numbers
\begin{equation*}
\frac{\sigma(n)}{d(n)}\leq\frac{n+1}{2}
\end{equation*}
which is due to E. S. Langford  \cite{ Mitrinovic}.
In 1987 S\'{a}ndor \cite{Sandor1987} proved the inequalities
\begin{equation*}
\frac{d(mn)}{d(m)\,d(n)}\leq \frac{\sigma(mn)}{\sigma(m)\,\sigma(n)}\,,
\end{equation*}
\begin{equation*}
\frac{(d(mn))^2}{d(m^2)\,d(n^2)}\frac{4mn}{(mn+1)^2}\leq\frac{(\sigma(mn))^2}{\sigma(m^2)\,\sigma(n^2)}\leq\frac{(d(mn))^2}{d(m^2)\,d(n^2)}
\end{equation*}
for all  $m, n = 1, 2, 3, \ldots$
The next year S\'{a}ndor \cite{Sandor1988} deduced that
\begin{equation*}
\frac{d(mn)}{d(m)}\geq\frac{\sigma(mn)}{n\sigma(m)}
\end{equation*}
for all  $m, n = 1, 2, 3, \ldots$

\section{The functions {\boldmath$\varphi(n)$} and {\boldmath$\psi(n)$}}
\indent

As a direct consequence of the formulas of $\varphi(n)$ and $\psi(n)$, we have
\begin{equation*}
\varphi(n)\leq\psi(n)\,.
\end{equation*}
In 1988 S\'{a}ndor \cite{Sandor1988a} proved the following inequalities
\begin{align*}
&\varphi\left(n\left[\frac{\psi(n)}{n}\right]\right)\leq n\,,\\
&\psi(n)^{\varphi(n)}<n^n  \quad \mbox{ for all } \quad n\geq2\,,\\
&\varphi(n)^{\psi(n)}>n^n  \quad \mbox{ if all prime factors of } n  \mbox{ are }\geq5\,,
\end{align*}
Further in 2008  Atanassov \cite{Atanassov2008} showed that
\begin{equation*}
\frac{6}{\pi^2}\varphi^2(n)>\psi(n)\sqrt{n}
\end{equation*}
for all odd numbers  except for the values $n=3, 5, 9, 15, 27, 45$, and
\begin{equation*}
\frac{18\sqrt{2}}{\pi^2}\varphi^2(n)>\psi(n)\sqrt{n}
\end{equation*}
for all even numbers $n>6$, and
\begin{equation*}
\varphi^3(n)>\psi^2(n)
\end{equation*}
for all $n\geq5$ except for the values $n=6, 8, 12, 16, 18, 24$, and
\begin{equation*}
\varphi^4(n)\geq\psi^3(n)
\end{equation*}
for all $n\geq5$ except for the values
$n=6, 8, 9, 12, 16, 18, 24, 32, 36, 48, 54, 72, 108, 144, 162, 192$, $216, 288, 324, 384, 432, 486, 576$.
In 2011 Sol\'{e} and Planat \cite{Sole} established the following theorem.
\begin{theorem}
For $n\geq2$ we have
\begin{equation*}
\frac{6}{\pi^2}<\frac{\varphi(n)\,\psi(n)}{n^2}<1 \,.
\end{equation*}
\end{theorem}

\begin{proof}
See (\cite{Sole}, Proposition 5).
\end{proof}
In the same year Atanassov \cite{Atanassov2011a} proved that for all natural numbers $n\geq2$ we have
\begin{equation*}
\varphi(n)+\psi(n)\geq2n\,.
\end{equation*}
Using the above inequality Atanassov \cite{Atanassov2011a} deduced that for all natural numbers $n\geq2$ we have
\begin{equation*}
\varphi(n)^{\varphi(n)}\psi(n)^{\psi(n)}>n^{2n}\,.
\end{equation*}
Further in 2013 Kannan and Srikanth \cite{Kannan} sharpened the above result by showing that
\begin{equation*}
\varphi(n)^{\varphi(n)}\psi(n)^{\psi(n)}>n^{\varphi(n)+\psi(n)}
\end{equation*}
for $n\geq2$. Afterwards in 2019 S\'{a}ndor and Atanassov \cite{Sandor2019} among other results
refined the above inequality, proving that for $n\geq2$ we have
\begin{align*}
n^{\varphi(n)+\psi(n)}<\left(\frac{\varphi(n)+\psi(n)}{2}\right)^{\varphi(n)+\psi(n)}&<\varphi(n)^{\varphi(n)}\psi(n)^{\psi(n)}
<\left(\frac{\varphi^2(n)+\psi^2(n)}{2}\right)^{\frac{\varphi(n)+\psi(n)}{2})}\\
&<\psi(n)^{\varphi(n)+\psi(n)}\,,\\
\left(\frac{\varphi(n)\psi(n)(\varphi(n)+\psi(n))}{\varphi^2(n)+\psi^2(n)}\right)^{\varphi(n)+\psi(n)}
&<\varphi(n)^{\psi(n)}\psi(n)^{\varphi(n)} <\left(\frac{2\varphi(n)\psi(n)}{\varphi(n)+\psi(n)}\right)^{\varphi(n)+\psi(n)}\\
&< (\varphi(n)\psi(n))^{\frac{\varphi(n)+\psi(n)}{2}}<n^{\varphi(n)+\psi(n)}\,.
\end{align*}
Subsequently in 2020 Alzer and Kwong \cite{Alzer2020} obtained that for $n\geq2$ we have
\begin{align*}
&\left(1-\frac{1}{n}\right)^{n-1}\left(1+\frac{1}{n}\right)^{n+1}
\leq\left(\frac{\varphi(n)}{n}\right)^{\varphi(n)}\left(\frac{\psi(n)}{n}\right)^{\psi(n)}\,,\\
&\left(\frac{\varphi(n)}{n}\right)^{\psi(n)}\left(\frac{\psi(n)}{n}\right)^{\varphi(n)}
\leq\left(1-\frac{1}{n}\right)^{n+1}\left(1+\frac{1}{n}\right)^{n-1}\,.
\end{align*}
Other results concerning inequalities with $\varphi(n)$ and $\psi(n)$ can be found in
\cite{Atanassov1996b}, \cite{Atanassov2011b}, \cite{Sandor1996}, \cite{Sandor2014}.

\section{The functions {\boldmath$\varphi(n)$} and {\boldmath$\pi(n)$}}
\indent

In 1951 Moser \cite{Moser} proved that if $n\geq90$, then
\begin{equation*}
\pi(n)\leq\varphi(n)\,.
\end{equation*}

\section{The functions {\boldmath$\psi(n)$} and {\boldmath$\sigma(n)$}}
\indent

It is easy to see that for $n\geq1$ we have
\begin{equation*}
\psi(n)\leq\sigma(n)\,.
\end{equation*}
In 2008 S\'{a}ndor \cite{Sandor2008b} showed that if $n\geq1$, then
\begin{equation*}
\sigma(n)<\frac{\pi^2}{6}\psi(n)\,.
\end{equation*}
Further in 2019 Atanassov and S\'{a}ndor \cite{Atanassov2019} deduced for $n\geq2$ the  inequality
\begin{equation*}
\sigma(n)^n<\psi(n)^{\sigma(n)}\,.
\end{equation*}
Other results concerning inequalities with $\psi(n)$ and $\sigma(n)$ can be found in Atanassov \cite{Atanassov2012}

\section{The functions {\boldmath$\pi(n)$} and {\boldmath$H_n$}}
\indent

In 2006 Hassani \cite{Hassani2006} showed that
\begin{equation*}
\frac{n}{H_n-\gamma-1+\frac{1}{4\log 3299}}<\pi(n)<\frac{n}{H_n-2-\frac{151}{100\log 7}}\,,
\end{equation*}
where the left inequality holds for $n\geq3299$ and the right inequality holds for $n\geq9$.

\section{The functions {\boldmath$\pi(n)$} and {\boldmath$\omega(n)$}}
\indent

In 2020 Zhang \cite{ZhangS} proved the following theorem.

\begin{theorem}
Let $n \geq 59$ be an integer. Then
\begin{equation*}
2\pi(n)-\pi(2n)\geq\omega(2n)\,.
\end{equation*}
\end{theorem}
\begin{proof}
See (\cite{ZhangS}, Theorem 1).
\end{proof}
The result of Zhang was improved by Alzer and Kwong \cite{Alzer2022a}.
\begin{theorem}
Let $n \geq 71$ be an integer. Then
\begin{equation*}
2\pi(n)-\pi(2n)\geq2\omega(n)\,,
\end{equation*}
with equality if and only if $n \in \{78, 100, 102, 126\}$.
\end{theorem}
\begin{proof}
See (\cite{Alzer2022a}, Theorem 1).
\end{proof}

\section{The functions {\boldmath$d(n)$}, {\boldmath$\omega(n)$} and {\boldmath$\Omega(n)$}}
\indent

It is well known (\cite{Hardy-Wright}, $\S$ 22.13) that the normal order of the divisor function satisfies
\begin{equation*}
2^{\omega(n)}\leq d(n)\leq2^{\Omega(n)} \,.
\end{equation*}
On the other hand using the definition of $d(n)$ and the arithmetic geometric mean inequality we deduce
\begin{equation*}
d(n)\leq \left(1+\frac{\Omega(n)}{\omega(n)}\right)^{\omega(n)}
\end{equation*}
for $n\geq2$.

In 1987 Somasundaram \cite{Somasundaram} showed that
\begin{equation*}
d(n)\leq \left(\log\left(n\prod_{p | n}p\right)^\frac{1}{\omega(n)}\right)^{\omega(n)}\prod_{p | n}\frac{1}{\log p}
\end{equation*}
for $n\geq2$.

In 1989 S\'{a}ndor (\cite{SandorHandbook}, II. 3, p. 41) proved that
\begin{equation*}
\frac{d(n^2)}{d(n)}\geq \left(\frac{3}{2}\right)^{\omega(n)}\,.
\end{equation*}
Other results concerning inequalities with $d(n)$, $\omega(n)$ and $\Omega(n)$ can be found in
\cite{Koninck}, \cite{Letendre}, \cite{Ramanujan}.

\section{The functions {\boldmath$\varphi(n)$}, {\boldmath$d(n)$} and {\boldmath$\sigma(n)$}}
\indent

Thanks to Makowski \cite{Makowski} from 1964 we have that if $n>1$, then
\begin{equation*}
\varphi(n)\,d^2(n)>\sigma(n)\,.
\end{equation*}
Afterwards in 1988 S\'{a}ndor \cite{Sandor1988a} proved that for all odd numbers
\begin{equation*}
\varphi(n)\,d(n)\geq \sigma(n)\,.
\end{equation*}

\section{The functions {\boldmath$\varphi(n)$}, {\boldmath$\psi(n)$} and {\boldmath$\sigma(n)$}}
\indent

It is well known that
\begin{equation*}
\varphi(n)\leq\psi(n)\leq\sigma(n)\,.
\end{equation*}
In 1989 S\'{a}ndor \cite{Sandor1989} proved that for all natural numbers
\begin{equation*}
\sigma(n)\leq \varphi(n)+d(n)(n-\varphi(n))\,.
\end{equation*}
Further in 2013 Atanassov \cite{Atanassov2013} formulated and proved one of the most elegant inequality
concerning the functions $\varphi(n)$, $\psi(n)$ and $\sigma(n)$.
Using induction on $\Omega(n)$ he proved that for every natural number $n\geq2$ we have
\begin{equation*}
\varphi(n)\,\psi(n)\,\sigma(n)\geq n^3+n^2-n-1\,.
\end{equation*}
In 2019 Atanassov and S\'{a}ndor \cite{Atanassov2019} derived for $n\geq2$ the inequality
\begin{equation*}
\psi(n)^n>\sigma(n)^{\varphi(n)}\,.
\end{equation*}
Afterwards in  2021 S\'{a}ndor and  Atanassov \cite{Sandor-Atanassov}, Theorem 1.2.1 showed that for $n\geq19$ we have
\begin{equation*}
\varphi(n)^{\psi(n)}>\sigma(n)^n\,.
\end{equation*}
Recently the author \cite{Dimitrov} showed that
\begin{align*}
&\varphi^2(n)+\psi^2(n)+\sigma^2(n)\geq 3n^2+2n+3\,,\\
&\varphi(n)\psi(n)+\varphi(n)\sigma(n)+\sigma(n)\psi(n)\geq 3n^2+2n-1
\end{align*}
for every natural number $n\geq2$.
Other results concerning inequalities with $\varphi(n)$, $\psi(n)$ and $\sigma(n)$ can be found in
\cite{Atanassov1995}, \cite{Atanassov1996a}, \cite{Atanassov2010}, \cite{Sandor2014}.

\section{The functions {\boldmath$\varphi(n)$}, {\boldmath$\psi(n)$}, {\boldmath$\sigma(n)$}, {\boldmath$\omega(n)$} and {\boldmath$\Omega(n)$}}
\indent

In 2014 Atanassov \cite{Atanassov2014} established that for every natural number $n\geq2$ we have
\begin{equation*}
\frac{\Omega(n)-\omega(n)}{2^{\Omega(n)-\omega(n)}}\varphi(n)<\sigma(n)-\psi(n)<2^{\Omega(n)-1}\varphi(n)\,.
\end{equation*}

\section{The Riemann hypothesis and the arithmetic functions}
\indent

The Riemann zeta function is a mathematical function of a complex variable defined as
\begin{equation*}
\zeta(s)=\sum _{n=1}^{\infty}\frac{1}{n^{s}}=\frac {1}{1^{s}}+\frac {1}{2^{s}}+\frac {1}{3^{s}}+\cdots
\end{equation*}
for $\operatorname{Re}(s)>1$. It is well known that its analytic continuation for $\operatorname{Re}(s)>0$ is given by the formula
\begin{equation*}
\zeta(s)=\sum\limits_{n=1}^\infty\frac{1}{n^s}+\frac{N^{1-s}}{s-1}-\frac{1}{2}N^{-s}
-s\int\limits_N^\infty\frac{\psi(t)}{t^{s+1}}dt\,,
\end{equation*}
where $N\geq1$, $\psi(t)=\{t\}-1/2$ and $\{t\}$ denotes the fractional part of $t$.
The Riemann zeta function has zeros at $-2, -4, -6, \ldots$ These are called the trivial zeros.
It is known that any non-trivial zero lies in the open strip
${\displaystyle \{s\in \mathbb {C} :0<\operatorname {Re} (s)<1\}}$, which is called the critical strip.
The set ${\displaystyle \{s\in \mathbb {C} :\operatorname {Re} (s)=1/2\}}$ is called the critical line.
The Riemann hypothesis asserts that all non-trivial zeros are on the critical line.
Many consider it to be the most important unsolved problem in pure mathematics
because its solution implies results about the distribution of prime numbers.
In 1915 Ramanujan showed that under the assumption of the Riemann hypothesis, the inequality
\begin{equation}\label{sigmaegamma}
\sigma(n)<e^\gamma n\log \log n
\end{equation}
holds for all sufficiently large $n$.
Afterwards in 1984 Robin \cite{Robin1984}  proved the following theorem.
\begin{theorem} (Robin's criterion)
The Riemann hypothesis is true if and only if the inequality \eqref{sigmaegamma}
holds for $n\geq 5041$.
\end{theorem}
\begin{proof}
See (\cite{Robin1984}, Theorem 1).
\end{proof}
In 2002 Lagarias \cite{Lagarias} showed that the Riemann hypothesis is equivalent to the following statement.
\begin{theorem}
The inequality
\begin{equation*}
\sigma(n)\leq H_n+e^{H_n}\log H_n
\end{equation*}
is true for every natural number $n$ if and only if the Riemann hypothesis is true.
\end{theorem}
\begin{proof}
See (\cite{Lagarias}, Theorem 1.1).
\end{proof}
The Riemann hypothesis implies a much tighter bound on the error in the estimate for
$\pi (x)$ and hence to a more regular distribution of prime numbers.
In this connection Schoenfeld \cite{Schoenfeld1976} has proved the next theorem.
\begin{theorem}\label{pilix}
 If the Riemann hypothesis holds then
\begin{align*}
&|\pi(x)-\operatorname{li}(x)|<\frac{\sqrt x\log x}{8\pi} \quad \text{ for } \quad  x\geq 2657\,\\
&\pi(x)-\operatorname{li}(x)<\frac{\sqrt x\log x}{8\pi} \quad \text{ for } \quad  x\geq 3/2\,.
\end{align*}
\end{theorem}
\begin{proof}
See (\cite{Schoenfeld1976}, Theorem 10, Corollary 1).
\end{proof}
In relation to Ramanujan's prime counting inequality \eqref{Ramanujan-inequality} in
2012 Hassani \cite{Hassani2012} established the following theorem.
\begin{theorem}
Assume that the Riemann hypothesis is true. Then the inequality \eqref{Ramanujan-inequality} is valid for
$x\geq138\; 766\; 146\; 692\; 471\; 228$.
\end{theorem}
\begin{proof}
See (\cite{Hassani2012}, Theorem 1.2).
\end{proof}
Subsequently in 2015 Dudek and Platt \cite{Dudek2015} refined Hassani's result by proving the next statement.
\begin{theorem}
Assume that the Riemann hypothesis is true. Then the inequality \eqref{Ramanujan-inequality} is valid for
$x\geq1.15\cdot10^{16}$.
\end{theorem}
\begin{proof}
See (\cite{Dudek2015}, Lemma 3.2).
\end{proof}
In 2018 Dusart \cite{Dusart2018a} improved the result in Theorem \ref{pilix}.
\begin{theorem}
 If the Riemann hypothesis holds then
\begin{equation*}
|\pi(x)-\operatorname{li}(x)|<\frac{\sqrt x}{8\pi}\log\left(\frac{x}{\log x}\right)
\end{equation*}
for $x\geq5639$.
\end{theorem}
\begin{proof}
See (\cite{Dusart2018a}, Proposition 2.6).
\end{proof}

\makeatletter
\renewcommand{\@biblabel}[1]{[#1]\hfill}
\makeatother

\end{document}